\documentclass[12pt]{amsart}
\pdfoutput=1
\usepackage{fullpage}
\usepackage{hyperref}
\usepackage{amssymb}
\usepackage{amstext}
\usepackage{amsmath}
\usepackage{amsthm}
\usepackage{stmaryrd}
\usepackage{verbatim}
\usepackage{enumerate}
\usepackage{mathtools}

\newtheorem{theorem}{Theorem}
\newtheorem{corollary}[theorem]{Corollary}
\newtheorem{proposition}[theorem]{Proposition}
\newtheorem{lemma}[theorem]{Lemma}
\theoremstyle{remark}
\newtheorem{remark}{Remark}
\newtheorem{example}{Example}
\newcommand{\SL}{\operatorname{SL}}
\newcommand{\spt}{\operatorname{spt}}
\newcommand{\sgn}{\operatorname{sgn}}
\newcommand{\D}{\operatorname{D}}
\newcommand{\hol}{\operatorname{hol}}
\newcommand{\Tr}{\operatorname{Tr}}
\begin{document}
\title{Recurrence and congruences for the smallest parts function}
\author{Wei Wang}
\address{Department of Mathematics, Shaoxing University, Shaoxing 312000, China}
\email{weiwang\_math@163.com}

\begin{abstract}
Let $\spt(n)$ be the number of smallest parts in the partitions of $n$. In this paper, we give some generalized Euler-like recursive formulas for the $\spt$ function in terms of Hecke trace of values of special twisted quadratic Dirichlet series. As a corollary, we give a closed form expression of the power series $\sum_{n\geq 0}\spt(\ell n-\delta_{\ell})q^n\pmod{\ell}$, $\delta_{\ell}:=(\ell^2-1)/24$, by Hecke traces for weight $\ell+1 $ cusp forms on $\SL_2(\mathbb{Z})$. We further establish an incongruence result for the $\spt$ function.
\end{abstract}
\thanks{ }
\maketitle
\section{Introduction}\label{section_intro}
Let $p(n)$ denote the number of partitions of $n$. Euler’s recurrence gives an efficient way to compute the partition numbers:
\begin{equation}\label{equat_euler-recurr}
\sum_k (-1)^kp(n-\omega(k))=0, \text{ for } n>0,
\end{equation}
where $\omega(k):=(3k^2+k)/2$ is the pentagonal number. Recently, Gomez, Ono, Saad, and Singh \cite{zbMATH08044761}  presented a generalization of (\ref{equat_euler-recurr}) using the theory of modular forms. They proved that for positive $n$, 
\begin{equation}\label{equat_euler-recurr-gen}
p(n)=\frac{1}{g_v(n,0)}\left(\alpha_v\sigma_{2v-1}(n)+\Tr_{2v}(n)+\sum_{k\in\mathbb{Z}\setminus \{0\}}(-1)^kg_{v}(n,k)p(n-\omega(k))\right),
\end{equation}
where $\sigma_{2v-1}(n)$ is the divisor function, $\Tr_{2v}(n)$ is the $n$-th weight $2v$ Hecke trace of values of special twisted quadratic Dirichlet series, and each $g_{v}(n,k)$ is a polynomial in $n$ and $k$. Later, Bhowmik, Tsai and Ye \cite{zbMATH08099054} gave similar Euler-like recursive formulas for the $t$-colored partition function when $t = 2$ or $t = 3$.

The smallest parts function $\spt(n)$, introduced by Andrews \cite{zbMATH05374875}, counts the number of smallest parts in the partitions of $n$.  For example, the partitions of $n = 5$ are (with the smallest parts underlined)
\begin{align*}
&\underline{5},~
4+\underline{1},~
3+\underline{2},~
3+\underline{1}+\underline{1},~
2+2+\underline{1},\\
&\underline{2}+\underline{1}+\underline{1}+\underline{1},~
\underline{1}+\underline{1}+\underline{1}+\underline{1}+\underline{1},
\end{align*}
and thus $\spt(5)=14$. This function not only plays a crucial role in combinatorics but also possesses rich arithmetic properties that have been investigated by many authors (see for example \cite{zbMATH05941220,zbMATH06191432,zbMATH06190321,zbMATH06133929}). The origin of its arithmetic properties is attributed to the connection between its generating function and harmonic Maass forms, see Section \ref{section_weak-Harm}. To state our results, we define the following mock Eisenstein series introduced by Zagier \cite{zbMATH05718024}:
\[
F_k:=\sum_{n\neq 0}(-1)^n\left(\frac{-3}{n-1}\right)n^{k-1}\frac{q^{n(n+1)/6}}{1-q^n}=-\sum_{r>s>0}	\left(\frac{12}{r^2-s^2}\right)s^{k-1}q^{rs/6}=\sum_{n\geq 1}a_{k}(n)q^n.
\]
We have the following Euler-like recurrence for $\spt(n)$ (see \cite[Theorem 1]{zbMATH06398837}):
\[
\sum_k(-1)^k\spt(n-\omega(k))=a_2(n), \text{ for } n>0.
\]
The primary goal of this paper is to generalize the above Euler-like recurrence, analogous to the generalization presented in (\ref{equat_euler-recurr-gen}). The differences between the $\spt$ function and the partition function lie in two key aspects: first, the generating function of the $\spt$ function is no longer a weakly holomorphic modular form but a harmonic Maass form, thus necessitating the employment of the technique of holomorphic projection. Second, the key step in the proof of (\ref{equat_euler-recurr-gen}) relies on expressing the generating function of the partition function as a Maass-Poincar\'e series, whereas in our case, the Maass-Poincar\'e series fails to converge,  and we shall adopt the method of analytic continuation to overcome this difficulty.

Before stating our results, we first introduce some notations. For a Hecke eigenform $f=\sum_{n\geq 1}a_f(n)q^n$ of weight $2k$ on $\SL_2(\mathbb{Z})$, we define its twisted quadratic Dirichlet series by 
\[
D_f:=\sum_{n\geq 1}\left(\frac{12}{n}\right) a_f\left(\frac{n^2-1}{24}\right)\left(\frac{1}{(n-1)^{2k-1}}-\frac1{(n+1)^{2k-1}}\right),
\]
(for the convergence of the series, see Section \ref{subsection_proof-inner-product}), and the weight $2k$ Hecke trace by
\[
\Tr_{2k}(n):=\frac{\Gamma(2k-1)}{(\pi/3)^{2k-1}}\sum_{f}a_f(n)\frac{D_f}{\langle f,f\rangle},
\]
where the sum runs over the normalized Hecke eigenforms of weight $2k$ and $\langle f,f\rangle$ denotes the Petersson norm of $f$. The Petersson inner product is defined by
\[
\langle f,g\rangle:=\int_{\SL_2(\mathbb{Z})\setminus{\mathbb{H}}}f(z)\overline{g(z)}y^{2k-2}dxdy,~z=x+iy,
\]
for modular forms $f$ and $g$ of weight $2k$ for $\SL_2(\mathbb{Z})$ with at least one of them a cusp form. Similar series have already appeared in prior works; see, for example, \cite{zbMATH07957369,zbMATH06534398}. We remark here that although the notation adopted here coincides with that in (\ref{equat_euler-recurr-gen}), their specific definitions are entirely different. 

Define the polynomials $q_v(n,k)$ by
\[
q_v(n,k)=2^{-2v}\cdot \sum_{0\leq j\leq v}\binom{2v+1}{2j+1}\left((6k+1)^2-24n\right)^j(6k+1)^{2v-2j}.
\]
The polynomials $q_v(n,k)$ are indeed the Chebyshev polynomials, they can be expressed by
\begin{equation}\label{equation_q-and-chebyshv}
q_v(n,k)=(6n)^vU_{2v}\left(\frac{6k+1}{\sqrt{24n}}\right). 
\end{equation}
Let $s(n):=12\spt(n)+(24n-1)p(n)$ (by convention, we agree that $\spt(0)=0$ and $p(0)=1$) and 
\[
\mathcal{S}_v(z):=\sum_{\substack{n\geq 0\\k\in\mathbb{Z}}}(-1)^{k+1}q_v(n,k)\cdot s(n-\omega(k))q^n,~q=e^{2\pi iz}.
\]
The main theorem of this paper is the following:
\begin{theorem}\label{theorem_main}
For every $v\geq 1$, we have
\[
\mathcal{S}_v(z)=E_{2v+2}(z)-12F_{2v+2}(z)+\frac{\Gamma(2v+1)}{(\pi/3)^{2v+1}}\sum_{f}\frac{D_f}{\langle f,f\rangle}f(z),
\]
where the sum runs over the normalized Hecke eigenforms of weight $2v+2$ and  $E_{2v+2}(z)$ is the weight $2v+2$ Eisenstein series.
\end{theorem}
Examining the coefficient of $q^n$ in Theorem \ref{theorem_main}, we obtain the following recurrence relation.
\begin{corollary}\label{coro_recurr}
For $n\geq 1$, we have
\begin{align*}
q_v(n,0) s(n)=&\sum_{k\in\mathbb{Z}\setminus\{0\}}(-1)^{k+1}q_v(n,k)s(n-\omega(k))\\
&+\frac{4v+4}{B_{2v+2}}\sigma_{2v+1}(n)+12a_{2v+2}(n)-\Tr_{2v+2}(n),
\end{align*}
where $B_n$ is the $n$-th Bernoulli number and $\sigma_{2v+1}(n)=\sum_{d\mid n}d^{2v+1}$.
\end{corollary}
Although Corollary \ref{coro_recurr} yields only a recurrence relation for $s(n)$,  for small values of $v$, using the convolution of divisor functions, we can obtain an explicit recurrence for $\spt(n)$.  We illustrate this with several examples below. For notational convenience, we set $\sigma_{2k-1}(0):=-\frac{B_{2k}}{4k}$.
\begin{example}
Corollary \ref{coro_recurr} with $v=1$ gives
\begin{align*}
(1-6n)\spt(n)=&\sum_{k\in\mathbb{Z}\setminus\{0\}}(-1)^{k+1}\left[(6k+1)^2-6n\right]\cdot\spt(n-\omega(k))\\
&-10\sigma_3(n)+24\sum_{\substack{i+j=n\\i,j\geq 0}}\sigma_1(i)\sigma_1(j)+a_4(n).
\end{align*}
Corollary \ref{coro_recurr} with $v=2$ gives
\begin{align*}
(36n^2-18n)\spt(n)=&\sum_{k\in\mathbb{Z}\setminus\{0\}}(-1)^{k+1}\left[36n^2-18n(6k+1)^2+(6k+1)^4\right]\cdot \spt(n-\omega(k))\\
&-576\sum_{\substack{i+j+k=n\\i,j,k\geq 0}}\sigma_1(i)\sigma_1(j)\sigma_1(k)+240\sum_{\substack{i+j=n\\i,j\geq 0}}\sigma_1(i)\sigma_3(j)+a_6(n).
\end{align*}
When $v=5$, the space of weight 12 cusp forms is one-dimensional, spanned by the Ramanujan $\Delta$ function. We find that
\begin{equation}\label{equation_Tr-and-tau}
\Tr_{12}(n)=-\frac{226437120}{691}\tau(n),
\end{equation}
and Corollary \ref{coro_recurr} with $v=5$ gives
\[
\tau(n)=\frac{691}{226437120}\sum_{k\in\mathbb{Z}}(-1)^kq_5(n,k)s(n-\omega(k))+\frac{1}{3456}\sigma_{11}(n)-\frac{691}{18869760}a_{12}(n).
\]
\end{example}
Using the relationship between the $\spt$ function and the second Atkin-Garvan moment, Andrews derived the following elegant congruences for the $\spt$ function:
\begin{equation}\label{equation_Andrews-con}
\begin{aligned}
&\spt\left(5n+4\right)\equiv 0\pmod 5,\\
&\spt(7n+5)\equiv 0\pmod 7,\\
&\spt(13n+6)\equiv 0\pmod {13}.
\end{aligned}
\end{equation}
These congruences are reminiscent of Ramanujan’s partition congruences
\begin{equation}\label{eqution_Ram-con}
\begin{aligned}
&p\left(5n+4\right)\equiv 0\pmod 5,\\
&p(7n+5)\equiv 0\pmod 7,\\
&p(11n+6)\equiv 0\pmod {11}.
\end{aligned}
\end{equation}
Using the recurrence relation (\ref{equat_euler-recurr-gen}), Bringmann, Craig and Ono \cite{zbMATH08123433} gave a closed form expression of the power series
\[
\sum_{n\geq 0}p(\ell n-\delta_{\ell})q^n \pmod{\ell},~\delta_{\ell}:=\frac{\ell^2-1}{24},
\]
by the generating function for the Hecke traces of $\ell$ ramified values of special Dirichlet series for weight $\ell-1$ cusp forms on $\SL_2(\mathbb{Z})$:
\[
\sum_{n\geq 1}p(\ell n-\delta_{\ell})q^n\equiv c_{\ell}\frac{\sum_{n\geq 1}\Tr_{\ell-1}(\ell n)q^n}{(q^{\ell};q^{\ell})_{\infty}} \pmod{\ell},
\]
for an explicit constant $c_{\ell}$. (We caution the reader that the symbol $\Tr$ in the preceding formula follows the convention of (\ref{equat_euler-recurr-gen}) and should not be confused with the notation for $\Tr$ used elsewhere in this paper.) Since there are no nontrivial cusp forms of weight 4, 6, and 10, this gives a new proof of Ramanujan’s congruences (\ref{eqution_Ram-con}). We establish an analogous version of the identity for the $\spt$ function as follows:
\begin{corollary}\label{corollary_modl}
Let $(q;q)_{\infty}:=\prod_{n\geq 1}(1-q^n)$. For every prime $\ell\geq 5$, we have
\[
\sum_{n\geq 1}\spt(\ell n-\delta_{\ell})q^n\equiv 12^{-1}\left(\frac{3}{\ell}\right)\frac{\sum_{n\geq 1}\Tr_{\ell+1}(\ell n)q^n}{(q^{\ell};q^{\ell})_{\infty}} \pmod{\ell}.
\]
\end{corollary}
Since there are no nontrivial cusp forms of weight 6, 8, and 14, the above result gives a new proof of Andrews' congruences (\ref{equation_Andrews-con}). The situation for other primes is also interesting. For example, when $\ell = 11$, by (\ref{equation_Tr-and-tau}) and $\Delta(z)=q(q;q)^{24}$, Corollary \ref{corollary_modl} yields
\[
\sum_{n\geq 0}\spt(11n+6)q^n\equiv 4(q;q)_{\infty}^{13}\pmod{11}.
\]
This gives a new proof of \cite[Theorem 6.1(6.3)]{zbMATH06190321}. The Ramanujan conjecture, asserting that simple congruences analogous to those in (\ref{eqution_Ram-con}) fail to hold for primes $\ell\geq 13$, was ultimately solved by Ahlgren and Boylan \cite{zbMATH02001026}. A similar conjecture exists for $\spt(n)$, supported by extensive numerical evidence (see \cite[Theorem 3]{zbMATH06865877}). It is further proved in \cite[Theorem 3]{zbMATH06865877} that the set of primes $\ell$ satisfying congruence has natural density zero. The conjecture can be reformulated in terms of Hecke traces as follows.
For prime $\ell\geq 5$, if
\[
\spt(\ell n-\delta_{\ell})\equiv 0\pmod {\ell} \text{, for every } n,
\]
then
\[
\Tr_{\ell+1}(n)\equiv 0\pmod {\ell} \text{, for every } n.
\]

Using a theorem of Garthwaite and Jameson \cite{zbMATH07282550}, we prove the following incongruence property for the $\spt$ function.
\begin{corollary}\label{coro_incon}
Suppose that $\ell\geq 5$ is prime and $t_0\in\{0,\ldots,\ell-1\}$. If $\left(\frac{1-24t_0}{\ell}\right)=-1$ and $\ell\nmid \spt(t_0)$, then for all $t\in\{0,\ldots,\ell-1\}$ such that 
\[
t\equiv t_0d^2+\frac{1-d^2}{24}\pmod{\ell},
\]
for some $d$ with $\gcd(d,6\ell)=1$, we have
\[
\sum_{n\geq 0}\spt(\ell n+t)q^n\not\equiv 0\pmod{\ell}.
\]
\end{corollary}
For example, choose $t_0=1$, then the above result implies that for all prime $\ell$ with $\left(\frac{-23}{\ell}\right)=-1$ and all $t$ satisfying
\[
t\equiv \frac{1+23d^2}{24}\pmod{\ell},~\gcd(d,6\ell)=1,
\]
(these $t$ range over all integers satisfying $\left(\frac{1-24t}{\ell}\right)=-1$) we have
\[
\sum_{n\geq 0}\spt(\ell n+t)q^n\not\equiv 0\pmod{\ell}.
\]

This paper is organized as follows. In Section \ref{section_weak-Harm}, we introduce the necessary notation and, in particular, recall the definition and basic properties of the harmonic Maass form. In Section \ref{subsection_holproj}, we explicitly compute the holomorphic projection of the harmonic Maass form associated with the spt function. In Section \ref{section_Maass-Poincare}, we construct the Maass-Poincar\'e series of weight $3/2$ and show that their holomorphic part is precisely the generating function for the spt function. The results stated in Section \ref{section_intro} are proved in Section \ref{section_proof}.

\section{Weak Harmonic Maass Forms}\label{section_weak-Harm}
In this section, we set notations and recall some fundamental facts regarding harmonic Maass forms. Let $f$ be a smooth function defined on the upper half-plane $\mathbb{H}$. For $\gamma=\left(\begin{smallmatrix}
a& b\\
c& d\\
\end{smallmatrix}\right)\in \operatorname{GL}^{+}_2(\mathbb{Q})$ and $k\in \frac{1}{2}\mathbb{Z}$, we define the weight $k$ slash operator $\big |_k$ by
\[
(f\big|_k\gamma)(z):=(\det \gamma)^{k/2}(cz+d)^{-k}f\left(\frac{az+b}{cz+d}\right).
\]
A smooth function $f:\mathbb{H}\rightarrow\mathbb{C}$ is called a weak harmonic Maass form of weight $k$ and multiplier $\chi$ for $\SL_2(\mathbb{Z})$ if it satisfies
\[
(f\big|_k\gamma)(z)=\chi(\gamma)f(z)\text{ for every }\gamma\in\SL_2(\mathbb{Z});
\]
and is annihilated by the weight $k$ hyperbolic Laplacian
\[
\Delta_k:=-y^2\left(\frac{\partial^2}{\partial x^2}+\frac{\partial^2}{\partial y^2}\right)+iky\left(\frac{\partial}{\partial x}+i\frac{\partial}{\partial y}\right),z=x+iy;
\]
and has the Fourier expansion of the form
\[
f(z)=\sum_{n\gg-\infty}c_f^{+}(n)q^{\frac{n}{m}}+c_f(0)y^{1-k}+\sum_{-n\ll\infty}c_f^{-}(-n)\Gamma\left(1-k,4\pi \frac{n}{m}y\right)q^{-\frac{n}{m}},q=e^{2\pi iz},
\]
where $\Gamma(1-k,y)$ denotes the incomplete gamma function and $m$ is defined by the identity $\chi\left(\left(\begin{smallmatrix}
1&1\\
0&1\\
\end{smallmatrix}\right)\right)=e\left(\frac{\alpha}{m}\right)$, $e(x):=e^{2\pi i x}$, $\alpha$ and $m$ co-prime. We denote the space of functions satisfying these properties by $H_k^{!}(\chi)$. We adopt the customary notation  $M_k^{!}(\chi),M_k(\chi),S_k(\chi)$ to denote the spaces of weakly holomorphic modular forms, modular forms and cusp forms, respectively. When $\chi$ is the trivial character, we omit it from the notation.

The Rankin-Cohen bracket for two functions $f$ and $g$, is defined by
\begin{equation}\label{equation_def-RC-bracket}
[f,g]_{v}:=\sum_{0\leq j\leq v} (-1)^j \binom{k+v-1}{v-j} \binom{l+v-1}{j}\D^jf\D^{v-j}g,
\end{equation}
where $\D f:=\frac{1}{2\pi i}\frac{df}{dz}=q\frac{df}{dq}$. It is known that (see \cite[Theorem 7.1]{zbMATH03485913})
\[
[f\big|_k\gamma,g\big|_l\gamma]_{v}=[f,g]_{v}\big|_{k+l+2v}\gamma,~\forall \gamma\in\SL_2(\mathbb{Z}).
\]
Thus if $f$ and $g$ transform like  modular forms of weights $k$ and $l$ with multiplier systems $\chi_1$ and $\chi_2$, respectively, then $[f,g]_{v}$ transforms like a modular form of weight $k+l+2v$ with multiplier system $\chi_1\chi_2$.

In this paper, we are interested in the multiplier system $\varepsilon$ which is attached to the Dedekind eta function $\eta$, defined by
\[
\eta(z)=q^{\frac{1}{24}}\prod_{n\geq 1}(1-q^n)=\sum_{n\geq 1}\chi_{12}(n)q^{\frac{n^2}{24}},
\]
where $\chi_{12}(n):=\left(\frac{12}{n}\right)$. Dedekind eta function has the transform property
\[
(\eta\big|_{1/2}\gamma)(z)=\varepsilon(\gamma)\eta(z),\forall \gamma\in \SL_2(\mathbb{Z}),
\]
which means $\eta\in M_{1/2}(\varepsilon)$. It is well-known that the eta multiplier $\varepsilon$ satisfies $\varepsilon\left(\left(\begin{smallmatrix}
1&1\\
0&1\\
\end{smallmatrix}\right)\right)=e(\frac{1}{24})$ and for $\gamma=\left(\begin{smallmatrix}
a& b\\
c& d\\
\end{smallmatrix}\right)$ with $c>0$,
\begin{equation}\label{equation_def_eta_multiper}
\varepsilon(\gamma)=\sqrt{-i}e^{-\pi is(d,c)}e\left(\frac{a+d}{24c}\right),
\end{equation}
where $s(d,c)$ is the Dedekind sum, see \cite[71.21]{zbMATH03399320}. 

Bringmann \cite{zbMATH05317179} showed that the generating function of the $\spt$ function, which is the primary focus of this paper, belongs to $H^{!}_{3/2}(\overline{\varepsilon})$ when supplemented with its non-holomorphic part, see also \cite[Theorem 2.1]{zbMATH06191432}. 
Let 
\begin{equation}\label{equation_def-alpha}
\alpha(z):=-\sum_{-1\leq n\equiv -1\bmod 24}\left(12\spt\left(\frac{n+1}{24}\right)+np\left(\frac{n+1}{24}\right)\right)q^{\frac{n}{24}},
\end{equation}
then the function
\begin{equation}\label{equation_genofspt}
\begin{aligned}
F(z)&=\alpha(z)+\frac{6\sqrt{3}i}{\pi}\int_{-\overline{z}}^{i\infty}\frac{\eta(w)}{(-i(w+z))^{3/2}}dw\\
&=\alpha(z)-\frac{3}{\sqrt{\pi}}\sum_{0>n\equiv -1\bmod 24}\chi_{12}(\sqrt{-n})|n|^{1/2}\Gamma\left(-\frac{1}{2},-\frac{\pi ny}{6}\right)q^{\frac{n}{24}},
\end{aligned}
\end{equation}
where $\chi_{12}(\sqrt{-n})$ is defined to be zero if $-n$ is not a square, belongs to the space $H^{!}_{3/2}(\overline{\varepsilon})$.

\section{Holomorphic Projection}\label{subsection_holproj}
As noted in the preceding section, our generating function for the spt function is a non-holomorphic modular form. In order to bring its analysis into the realm of holomorphic modular forms, we apply the method of holomorphic projection. Suppose we have a weight $k$ real-analytic modular form $\tilde{f}$ with moderate growth at the cusps. Then this function defines a linear functional on the space of weight $k$ holomorphic cusp forms by $g\rightarrow \langle g,\tilde{f}\rangle$, this functional must be given by $\langle \cdot,f\rangle$ for a holomorphic cusp form $f$ of the same weight. This $f$ is essentially the
 holomorphic projection of $\tilde{f}$, which we defined as $\pi_{\hol}(\tilde{f})$. We recall some basic properties about holomorphic projection (see \cite[10.1]{zbMATH06828732}). 
\begin{proposition}\label{prop_holoproj}
Let $f:\mathbb{H}\rightarrow \mathbb{C}$ be a smooth function  transforming like a modular form 
of weight  $k \geq 2$ on $\SL_2(\mathbb{Z})$ with Fourier expansion
\[
f(z)=\sum_{n\in \mathbb{Z}}a_f(n,y)q^n,
\]
and assume that $f$ satisfies the following:
\begin{enumerate}[\rm (i).]
\item  $f(z)=c_0+O(y^{-\delta})$ for some $\delta>0$ as $y\rightarrow \infty$.

\item  $a_f(n,y)=O(y^{1-k+\epsilon})$ for some $\epsilon>0$ as $y\rightarrow \infty$ for all $n$.
\end{enumerate}
The following are true.
\begin{enumerate}[\rm (1).]
\item  $\pi_{\hol}(\tilde{f})=c_0+\sum_{n\geq 1}c(n)q^n,$
where 
\[
c(n)=\frac{(4\pi n)^{k-1}}{\Gamma(k-1)}\int_0^{\infty}a_f(n,y)e^{-4\pi ny}y^{k-2}dy.
\]

\item  If $k=2$, $\pi_{\hol}(f)$ is a quasi-modular form and for $k\geq 4$, $\pi_{\hol}(f)$ is a modular form.

\item  $\langle f,g\rangle=\langle \pi_{\hol}(f),g\rangle$ for every $g\in S_k$.
\end{enumerate}
\end{proposition}

For $k\geq 1$, we recall the mock Eisenstein series as
\[
F_k:=\sum_{n\neq 0}(-1)^n\left(\frac{-3}{n-1}\right)n^{k-1}\frac{q^{n(n+1)/6}}{1-q^n}=-\sum_{r>s>0}	\left(\frac{12}{r^2-s^2}\right)s^{k-1}q^{rs/6}.
\]
The following result was already stated (without proof) by Zagier \cite{zbMATH05718024} and given a brief proof (without explicit computation) by Ahlgren and Kim \cite{zbMATH06865877}. For the sake of completeness, we provide the detailed computation here.
\begin{proposition}\label{prop_holoprojofspt}
Let $F$ be the function defined in \textnormal{(\ref{equation_genofspt})}, then
\[
\pi_{\hol}(F\eta)=\alpha\eta+12F_{2}=E_2,
\]
and for $v\geq 1$,
\[
\pi_{\hol}([F,\eta]_v)=[\alpha,\eta]_v+24^{-v}\binom{2v}{v}\cdot 12F_{2v+2}\in M_{2v+2}.
\]
\end{proposition}
\begin{proof}
Expressing $F$ as $\alpha+NH$, the holomorphic projection of $[F,\eta]_v$ takes the form 
\[
[\alpha(z),\eta(z)]_v+\pi_{\hol}([NH(z),\eta(z)]_v).
\]
To evaluate $\pi_{\hol}([NH(z),\eta(z)]_v)$, we make use of a result of Mertens \cite{zbMATH06620622}, which combines Theorem 4.6 and Proposition 5.3 of \cite{zbMATH06620622}. We state it as follows: 
\[
\text{ If }f^{-}=\sum_{n\geq 0}c_{f}(n)n^{\frac{1}{2}}\Gamma(-\tfrac{1}{2},4\pi ny)q^{-n} \text{ and } g=\sum_{n\geq 0}a_g(n)q^n,
\]
then
\[
\pi_{\hol}([f^{-},g]_v)= 2^{-2v+1}\binom{2v}{v}\sqrt{\pi}\sum_{N\geq 1}\sum_{m-n=N}(m^{\frac{1}{2}}-n^{\frac{1}{2}})^{2v+1}a_g(m)c_f(n)q^N.
\]
In order to align with the assumptions of the aforementioned result, we proceed to evaluate $\pi_{\hol}([NH(24z),\eta(24z)]_v)$. Using the fact that
\[
a_g(m)=\chi_{12}(\sqrt{m}) \text{ and }c_f(n)=-\frac{3}{\sqrt{\pi}}\chi_{12}(\sqrt{n}),
\]
for $g=\eta(24z)$ and $f=NH(24z)$, we obtain
\[
\pi_{\hol}([NH(24z),\eta(24z)]_v)=-3\cdot 2^{-2v+1}\binom{2v}{v}\sum_{N\geq 1}\sum_{m^2-n^2=N}(m-n)^{2v+1}\left(\frac{12}{mn}\right)q^N.
\]
In the above identity, replace $(m+n)/2$ by $r$ and $(m-n)/2$ by $s$. The right-hand side then becomes $ \binom{2v}{v}12\cdot F_{2v+2}(24z)$. Finally, observing that
\[
[NH(24z),\eta(24z)]_v=24^{-v}\left([NH,\eta]_v\right)(24z),
\]
we complete the proof of the result.
\end{proof}

\section{Maass-Poincar\'e series with eta multiplier}\label{section_Maass-Poincare}
 To define the Maass-Poincar\'e series, we first recall definitions of the Whittaker functions $M_{\mu,\nu}(y)$ and $W_{\mu,\nu}(y)$, they are linearly independent solutions of the differential equation
\begin{equation}\label{equation_Whittaker_equation}
\frac{d^2w}{dy^2}+\left(-\frac{1}{4}+\frac{\mu}{y}+\frac{\frac{1}{4}-\nu^2}{y^2}\right)w=0.
\end{equation}
At the special point $s=k/2$, we have
\[
y^{-\frac{k}{2}}M_{\sgn(n)\tfrac k2,\tfrac{k-1}{2}}(y)=e^{-\sgn(n)\frac{y}{2}},
\]
and
\begin{equation}\label{equtaion_W-function-special}
y^{-\frac{k}{2}}W_{\sgn(n)\tfrac k2,\tfrac{k-1}{2}}(y)=\begin{dcases}
e^{-\frac{y}{2}}	&\text{ if } n>0, \\ 
\Gamma(1-k,y)e^{\frac{y}{2}}	&\text{ if } n<0.
\end{dcases}
\end{equation}
For $m\equiv -1\pmod {24}$ and $s\in\mathbb{C}$, define
\[
\varphi_{m,s}(z):=\left(4\pi |m|\frac{y}{24}\right)^{-\frac{k}{2}}M_{\sgn(m)\frac{k}{2},s-\frac{1}{2}}\left(4\pi |m|\frac{y}{24}\right)e\left(\frac{mx}{24}\right).
\]
It follows from (\ref{equation_Whittaker_equation}) that $\varphi_{m,s}$ is an eigenfunction of the weight $k$ hyperbolic Laplacian:
\begin{equation}\label{equation_laplace-of-phi}
\Delta_k\varphi_{m,s}(z)=\left(s-\frac{k}{2}\right)\left(1-\frac{k}{2}-s\right)\varphi_{m,s}(z),
\end{equation}
and due to the asymptotic behaviour of the Whittaker function, we have
\[
\varphi_{m,s}(z)=O(y^{\operatorname{Re}(s)-k/2}), \text{ as }y\rightarrow 0.
\]
In this subsection, we focus primarily on functions in $H_{k}(\overline{\varepsilon})$. The space $H_{k}(\overline{\varepsilon})$ is trivial except when $2k\equiv 3\pmod4$. Throughout the remainder of the paper, we therefore assume that $2k\equiv 3\pmod4$. We now define the Maass-Poincar\'e series with multiplier $\overline{\varepsilon}$ as 
\begin{equation}\label{equat_def-of-poincare}
P_{m,k,s}(z)=\sum_{\gamma\in \Gamma_{\infty}\setminus \SL_2(\mathbb{Z})}\varepsilon(\gamma)\varphi_{m,s}\big|_k\gamma.
\end{equation}
When $\Re(s)>1$, the series defined above converges absolutely and uniformly on
compact subsets of $\mathbb{H}$. To compute their Fourier coefficients, we first make some notational conventions. Define the Kloosterman sum as 
\[
K(m,n;c):=\sum_{d(c^{*})}e^{\pi is(d,c)}e\left(\frac{\overline{d}m+dn}{c}\right).
\]
Here $c^{*}$ indicates that the sum is restricted to residue classes coprime to $c$, and $\overline{d}$ denotes the inverse of $d$ modulo $c$. Let $J_\alpha(x)$ and $I_\alpha(x)$ be the usual Bessel functions.
\begin{proposition}\label{prop_fourier-poincare}
Let $k$ and $P_{m,k,s}(z)$ be as above, and suppose that $m\equiv -1\pmod{24}$ and $\textup{Re}(s)>1$.  Then we have
\[
P_{m,k,s}(z)= \varphi_{m,s}(z) + \sum_{n \equiv -1(24)} g_{m,n}(s)L_{m,n}(s)\left(4\pi |m|\frac{y}{24}\right)^{-\frac k2} W_{\sgn(n)\frac k2,s-\frac 12}\left(4\pi |n| \frac{y}{24}\right) e\left(\frac{nx}{24}\right),
\]
where
\[
g_{m,n}(s) := \frac{2\pi i^{-k-\frac 12} \Gamma(2s) \sqrt{|m/n|} }{\Gamma\left(s+\sgn(n)\frac k2\right)}
\]
and 
\[ 
	L_{m,n}(s) := 
		\begin{dcases}
			\sum\limits_{c>0} \frac{K(-m',-n';c)}{c} J_{2s-1}\left(\frac{\pi\sqrt{|mn|}}{6c}\right) &\text{ if } mn>0, \\ 
			\sum\limits_{c>0} \frac{K(-m',-n';c)}{c} I_{2s-1}\left(\frac{\pi\sqrt{|mn|}}{6c}\right) &\text{ if } mn<0,
		\end{dcases}
\]
where $m=24m'-1$ and $n=24n'-1$.
\end{proposition}
\begin{proof}
As the function $P_m(z,s)-\varphi_{m,s}(z)$ has polynomial growth as $y\rightarrow\infty$, and $e(\tfrac{x}{24})(P_m(z,s)-\varphi_{m,s}(z))$ is periodic with period 1, we have the Fourier expansion 
\[
e\left(\frac{x}{24}\right)(P_m(z,s)-\varphi_{m,s}(z))=\sum_{n\in\mathbb{Z}}a(n,y)e(nx)
\]
with
\[
a(n,y)=\int_0^{\infty}e\left(\frac{x}{24}\right)(P_m(z,s)-\varphi_{m,s}(z))e(-nx)dx.
\]
The Fourier coefficients can be computed as
\begin{align*}
a(n,y)=&\sum_{\substack{c>0\\d(c^{*})}}\frac{\varepsilon(\gamma)}{c^k}e\left(\frac{\overline{d}m+d(24n-1)}{24c}\right)\\
&\times\int_{-\infty}^{+\infty}z^{-k}\varphi_{m,s}\left(\frac{y}{24c^2|z|^2}\right)e\left(-\frac{mx}{24c^2|z|^2}-(n-\tfrac1{24})x\right)dx.
\end{align*}
Write $m=24m'-1$, and by (\ref{equation_def_eta_multiper}) we see 
\[
\sum_{\substack{c>0\\d(c^{*})}}\varepsilon(\gamma)e\left(\frac{\overline{d}m+d(24n-1)}{24c}\right)=\sqrt{-i}K(-m',-n,c),
\]
since the Kloosterman sum is real. The calculation of the integrals in the Fourier series is standard, following the standard method in \cite[Proposition 8]{zbMATH06516390}.
\end{proof}
In the main case we care about ($k=\frac 32, s=\frac 43$), the series (\ref{equat_def-of-poincare}) no longer converges. However we show in the following that they admit an analytic continuation at $s=3/4$.
\begin{proposition}\label{prop_poincare-analytic}
For every $m<0$ with $m\equiv-1\pmod{24}$, the function $P_{m,\frac{3}{2},s}$ admits an analytic continuation at $s=3/4$. If $-m$ is not a square, then $P_{m,\frac{3}{2},\frac{3}{4}}\in M^{!}_{3/2}(\overline{\varepsilon})$ and if $-m$ is a square, then $P_{m,\frac{3}{2},\frac{3}{4}}\in H^{!}_{3/2}(\overline{\varepsilon})$ and 
\[
\chi_{12}(\sqrt{-m})P_{m,\frac{3}{2},\frac{3}{4}}-P_{-1,\frac{3}{2},\frac{3}{4}}\in M^{!}_{3/2}(\overline{\varepsilon}).
\]
\end{proposition}
\begin{proof}
From Proposition \ref{prop_fourier-poincare}, we have the Fourier expansion
\[
P_{m,\frac{3}{2},s}=\varphi_{m,s}(z)+\sum_{n\equiv -1(24)}a_m(n,s)\mathcal{W}_n(y,s)e\left(\frac{nx}{24}\right)
\]
with
\[
\mathcal{W}_n(y,s)=-\frac{2\pi \Gamma(2s)}{\Gamma\left(s+\sgn(n)\frac 34\right)}\left|\frac{n}{m}\right|^{\frac{1}{4}}\left(\frac{\pi |n|y}{6}\right)^{-\frac 34} W_{\sgn(n)\frac 34,s-\frac 12}\left( \frac{\pi |n|y}{6}\right)
\]
and
\[
a_m(n,s)=\sum\limits_{c>0} \frac{K(-m',-n';c)}{c}\times 
\begin{dcases}
			 J_{2s-1}\left(\frac{\pi\sqrt{|mn|}}{6c}\right) &\text{ if } n<0, \\ 
			 I_{2s-1}\left(\frac{\pi\sqrt{|mn|}}{6c}\right) &\text{ if } n>0,
		\end{dcases}
\]
where $m=24m'-1$ and $n=24n'-1$.

Following the approach used in the proof of \cite[Lemma 5]{zbMATH06951117}, combined with the  estimates of Kloosterman sum established in \cite[Theorem 8]{zbMATH06530917}, the $a_m(n,y)$ satisfy the following.

Fix $m<0$, there exists a positive constant $A$ such that the following are true.
\begin{enumerate}[(1)]
\item For $n>0$, we have
\[
a_m(n,s)=O\left(n^A \exp\left(\frac{\pi\sqrt{n}}{6}\right)\right)
\]
uniformly for $s\in[\tfrac 34,1]$.
\item For $n<0$, we have
\[
a_m(n,s)=\chi_{12}(\sqrt{-n})\chi_{12}(\sqrt{-m})\frac{3}{\pi^2}\cdot\frac{(-n)^{\frac{1}{4}}}{s-\frac{3}{4}}+O(n^A)
\]
uniformly for $s\in(\tfrac 34,1]$ if both $-n$ and $-m$ are square, and for $s\in[\tfrac 34,1]$ if neither $-n$ nor $-m$ is a  square.
\end{enumerate}
We infer that when $-m$ is not a perfect square, the Fourier coefficients $a(n,s)\mathcal{W}_n(y,s)$ are holomorphic at $s=\frac{3}{4}$, and the non-holomorphic part of  $P_{m,\frac{3}{2},s}$ vanishes at $s=\frac{3}{4}$ due to the zero of $\Gamma(s-\frac 34)^{-1}$. Thus in this case, $P_{m,\frac{3}{2},\frac{3}{4}}$ is a well-defined weakly holomorphic modular form. When $-m$ is a perfect square, $a_m(n,s)$ has a simple pole at $s=\frac{3}{4}$ but in this case, the singularity cancels with the zero of $\Gamma(s-\frac 34)^{-1}$. In this case, $P_{m,\frac{3}{2},\frac{3}{4}}$ is a harmonic Maass form with non-zero shadow. Moreover, these shadows differ only by a scalar factor $\chi_{12}(\sqrt{-m})$. This establishes the second assertion of the proposition.
\end{proof}

\begin{proposition}\label{prop_spt-poincare}
Recall the function defined in \textnormal{(\ref{equation_genofspt})}, we have
\[
F(z)=P_{-1,\frac{3}{2},\frac{3}{4}}(z).
\]
\end{proposition}
\begin{proof}
It follows from the proof of Proposition \ref{prop_poincare-analytic} and (\ref{equtaion_W-function-special}) that the non-holomorphic part of $P_{-1,\frac{3}{2},\frac{3}{4}}$ equals
\[
-\frac{3}{\sqrt{\pi}}\sum_{0>n\equiv -1(24)}\chi_{12}(\sqrt{-n})|n|^{\frac{1}{2}}\Gamma\left(-\frac{1}{2},-\frac{\pi ny}{6}\right)q^{\frac{n}{24}}.
\]
Comparing with the Fourier expansion of $F$, we conclude that $F$ and $P_{-1,\frac{3}{2},\frac{3}{4}}$ share the same non-holomorphic part, as well as the same principal part $q^{-\frac{1}{24}}$. Thus the function $\eta(z)(F(z)-P_{-1,\frac{3}{2},\frac{3}{4}}(z))$ is a modular form of weight 2. However, since there is no non-zero holomorphic modular form of weight 2 for $\SL_2(\mathbb{Z})$,  it must be zero.
\end{proof}

\begin{remark}
Ahlgren and Andersen \cite{zbMATH06530917} using the theta lift of Bruinier and Funke \cite{zbMATH05030844} proved an exact formula for the  $\spt$ function:
\[
\spt(n)=\frac{\pi}{6}(24n-1)^{\frac{1}{4}}\sum_{c=1}^{\infty}\frac{A_c(n)}{c}(I_{3/2}-I_{1/2})\left(\frac{\pi\sqrt{24n-1}}{6c}\right),
\]
where $A_c(n):=K(0,-n;c)$. Proposition \ref{prop_spt-poincare} essentially provides an alternative proof of the above identity, relying solely on estimates of the Kloosterman sum.
\end{remark}

\begin{remark}
Jeon, Kang and Kim in \cite[Proposition 5.1]{zbMATH06283160} constructed the Maass-Poincar\'e series of weight $1/2$ with theta multiplier in Kohnen's plus space. It is noteworthy that the Maass-Poincar\'e series with theta multiplier exhibit a formal resemblance to the series constructed in the present paper. They further establish a duality between weight $3/2$ series and the weight $1/2$  series constructed by Duke, Imamo\={g}lu and T\'oth \cite{zbMATH05960675}. In fact, an analogous duality holds in our setting as well. The Maass-Poincar\'e series for general $m$ constructed in Proposition \ref{prop_spt-poincare} coincides with the mock modular grid introduced by Ahlgren and Kim in \cite{zbMATH06342625}, where they also established a duality between the mock modular grid and the one with weakly holomorphic modular forms of weight $1/2$. As we can now observe, this is essentially the duality between the Maass-Poincar\'e series of weights $1/2$ and $3/2$. Such phenomena are quite common, for instance, see \cite{zbMATH05271017,zbMATH02078167}.
\end{remark}

\section{Proof of the results}\label{section_proof}
The core result to be proven in this section is as follows: 
\begin{theorem}\label{theorem_inner-product}
For any Hecke eigenform $f$ of weight $2v+2$ on $\SL_2(\mathbb{Z})$, we have  
\[
\langle [F,\eta]_v,f \rangle=24^{-v}\binom{2v}{v}\frac{\Gamma(1+2v)}{(\pi/3)^{1+2v}}\cdot D_f.
\]
\end{theorem}
We first apply this result to establish Theorem \ref{theorem_main}, Corollary \ref{corollary_modl} and Corollary \ref{coro_incon}, and the proof of Theorem \ref{theorem_inner-product} is deferred to Section \ref{subsection_proof-inner-product}.
\subsection{Proof of Theorem 1}
\begin{lemma}\label{lemma_RC-alpha-eta}
Recall the definition of $\alpha(z)$ in (\ref{equation_def-alpha}). Regarding it as a form of weight $3/2$, we have 
\[
[\alpha,\eta]_v=24^{-v}\binom{2v}{v}\mathcal{S}_v.
\]
\end{lemma}
\begin{proof}
By the definition of Rankin-Cohen bracket (\ref{equation_def-RC-bracket}) and 
\[
\eta(z)=\sum_{k\in\mathbb{Z}}(-1)^kq^{\omega(k)+\frac{1}{24}},
\]
we obtain
\begin{equation}\label{equation_exp-[a,e]}
\begin{aligned}
[\alpha,\eta]_v=-&\sum_{0\leq j\leq v}(-1)^j\binom{v+\frac{1}{2}}{v-j}\binom{v-\frac{1}{2}}{j}\\
&\times \left(\sum_{n\geq 0}(n-\tfrac 1{24})^js(n)q^{n-\frac{1}{24}}\right)\left(\sum_{k\in\mathbb{Z}}(\omega(k)+\tfrac{1}{24})^{v-j}(-1)^kq^{\omega(k)+\frac{1}{24}}\right).
\end{aligned}
\end{equation}
On the other hand, we have
\begin{equation}\label{equation_com-ident}
\binom{v+\frac{1}{2}}{v-j}\binom{v-\frac{1}{2}}{j}=2^{-2v}\binom{2v}{v}\binom{2v+1}{2j+1},
\end{equation}
Substituting it into (\ref{equation_exp-[a,e]}), and performing some straightforward calculations yields the desired result.
\end{proof}
\begin{proof}[Proof of Theorem \ref{theorem_main}]
By Proposition \ref{prop_holoprojofspt}, the holomorphic projection of $[F,\eta]_v$ is a modular form of weight $2v+2$, we can write it as 
\begin{equation}\label{equation_holo-express}
\pi_{\hol}([F,\eta]_v)=[\alpha,\eta]_v+24^{-v}\binom{2v}{v}\cdot 12F_{2v+2}=24^{-v}\binom{2v}{v}E_{2v+2}+\sum_{f}\frac{\langle\pi_{\hol}([F,\eta]_v),f \rangle}{\langle f,f\rangle}f,
\end{equation}
where the sum runs over the normalized Hecke eigenforms of weight $2v+2$. The coefficient of $E_{2v+2}$ is precisely the constant term of $[\alpha,\eta]_v$, whose explicit value can be deduced from Lemma \ref{lemma_RC-alpha-eta}. On the other hand, by the property of holomorphic projection and Theorem \ref{theorem_inner-product}, for any normalized eigenform $f$ of weight $2v+2$, we have
\begin{equation}\label{equation_inner-product}
\langle \pi_{\hol}([F,\eta]_v),f\rangle=\langle [F,\eta]_v,f\rangle=24^{-v}\binom{2v}{v}\frac{\Gamma(1+2v)}{(\pi/3)^{1+2v}}\cdot D_f.
\end{equation}
Theorem \ref{theorem_main} now follows by combining (\ref{equation_holo-express}), (\ref{equation_inner-product}), and Lemma \ref{lemma_RC-alpha-eta}.
\end{proof}

The examples following Corollary \ref{coro_recurr} are obtained as follows: by
\[
\frac{E_2}{\eta}=-24\D\left(\frac{1}{\eta}\right)=\sum_{n\geq 0}(1-24n)p(n)q^{n-\frac{1}{24}},
\]
isolating the $\spt(n)$ and $p(n)$ terms and using the expression of Theorem \ref{theorem_main}, we arrive at
\[
\mathcal{S}^{\prime}_{v}=E_{2v+2}-12F_{2v+2}+\frac{\Gamma(1+2v)}{(\pi/3)^{1+2v}}\sum_{f} \frac{D_f}{\langle f,f\rangle}f-24^v\binom{2v}{v}^{-1}\cdot [E_2\cdot \eta^{-1},\eta]_v,
\]
where 
\[
\mathcal{S}^{\prime}_{v}=\sum_{\substack{n\geq 0\\k\in\mathbb{Z}}}(-1)^{k+1}q_v(n,k)\cdot 12\spt(n-\omega(k))q^n.
\]
By mathematical induction, we conclude that the function $[E_2\cdot \eta^{-1},\eta]_v$ is a sum of quasi-modular forms of weight less than $2v+2$. The structure theorem for the ring of quasi-modular forms (see \cite[Section 5.1]{zbMATH05808162}) now implies that $[E_2\cdot \eta^{-1},\eta]_v\in \mathbb{Q}[E_2,E_4,E_6]$. Using the method of undetermined coefficients, for small values of $v$, we express $[E_2\cdot \eta^{-1},\eta]_v$ explicitly:
\begin{align*}
[E_2\cdot \eta^{-1},\eta]_1&=\frac{1}{24}E_4+\frac{1}{24}E_2^2,\\
[E_2\cdot \eta^{-1},\eta]_2&=\frac{1}{96}E_6-\frac{1}{192}E_4E_2+\frac{1}{192}E_2^3.
\end{align*}
The identities above furnish precisely the cases $v=1$ and $v=2$ in the examples.
\subsection{Proof of Corollary \ref{corollary_modl}} Using the congruence properties of the Chebyshev polynomials 
\[
U_{\frac{\ell-1}{2}}(x)\equiv (x^2-1)^{(\ell-1)/2}\equiv \left(\frac{x^2-1}{\ell}\right)\pmod{\ell}
\]
and (\ref{equation_q-and-chebyshv}), we deduce that
\[
q_{\frac{\ell-1}{2}}(n,k)\equiv \left(\frac{(6k+1)^2-24n}{\ell}\right)\pmod{\ell},
\]
and thus
\begin{equation}\label{equation_congurence-S_l-1}
\mathcal{S}_{\frac{\ell-1}{2}}\equiv \sum_{\substack{n\geq 0\\ k\in\mathbb{Z}}}(-1)^{k+1}\left(\frac{(6k+1)^2-24n}{\ell}\right)s(n-\omega(k))q^n.
\end{equation}
On the other hand, by the definition of $F_k$, we have $F_{\ell+1}\equiv F_{2}\pmod {\ell}$ and by Kummer’s congruences, we have $E_{\ell+1}\equiv E_2\pmod{\ell}$. From these facts, we obtain that
\begin{equation}\label{equation_cusppart}
\begin{aligned}
\mathcal{S}_{\frac{\ell-1}{2}}&\equiv E_2-12F_2+\frac{\Gamma(\ell)}{(\pi/3)^{\ell}}\sum_{f}\frac{D_f}{\langle f,f\rangle}f \\
&=\mathcal{S}_0 +\frac{\Gamma(\ell)}{(\pi/3)^{\ell}}\sum_{f}\frac{D_f}{\langle f,f\rangle}f \pmod{\ell}.
\end{aligned}
\end{equation}
Considering the coefficient of $q^{\ell n}$ in the preceding series, combined with (\ref{equation_congurence-S_l-1}), yields
\begin{equation}\label{equation_s-and-Df}
\begin{aligned}
\Tr_{\ell+1}(\ell n)&\equiv  \sum_{k\in\mathbb{Z}}(-1)^{k+1}\left(\left(\frac{(6k+1)^2-24\ell n}{\ell}\right)-1\right)s(\ell n-\omega(k))\\
&\equiv 12\sum_{6k\equiv-1\bmod{\ell}}(-1)^k \spt(\ell n-\omega(k))\pmod{\ell}.
\end{aligned}
\end{equation}
We also have 
\begin{equation}\label{equation_sumspt}
\sum_{n\geq 0}\sum_{6k\equiv -1\bmod{\ell}}(-1)^k\spt(n-\omega(k))q^n=\left(\sum_{n\geq 0}\spt(n)q^n\right)\cdot \left(\sum_{6k\equiv -1\bmod{\ell}}(-1)^kq^{\omega(k)}\right).
\end{equation}
On the other hand, we have
\begin{align*}
\sum_{6k\equiv -1\bmod{\ell}}(-1)^kq^{\omega(k)}&=q^{-1/24}\sum_{n\geq 1}\left(\frac{12}{\ell n}\right)q^{\frac{\ell^2n^2}{24}}\\
&=\left(\frac{3}{\ell}\right)\sum_{n\geq 1}\left(\frac{12}{\ell}\right)q^{\frac{\ell^2n^2}{24}}=\left(\frac{3}{\ell}\right)q^{\frac{\ell^2-1}{24}}(q^{\ell^2};q^{\ell^2})_{\infty}.
\end{align*}
Substituting the above expression into \eqref{equation_sumspt} and applying the operator $U_{\ell	}$ to both sides of the resulting equation, by utilizing the properties of $U_{\ell}$:
\[
U_{\ell}(f(q^{\ell})g(q))=f(q)\cdot U_{\ell}(g(q)),
\]
we obtain
\[
\sum_{n\geq 0}\sum_{6k\equiv-1\bmod{\ell}}(-1)^k \spt(\ell n-\omega(k))q^n\equiv \left(\frac{3}{\ell}\right)(q^{\ell};q^{\ell})_{\infty}\sum_{n\geq 1}\spt(\ell n-\delta_{\ell})q^n\pmod{\ell},
\]
which completes the proof.

The remark  following Corollary \ref{corollary_modl} can be deduced from \cite[Proposition 6]{zbMATH06865877}. The function $f_{\ell}$ appearing therein is precisely $-\mathcal{S}_{\frac{\ell-1}{2}}+\mathcal{S}_0$.

\subsection{Proof of Corollary \ref{coro_incon}}\label{subsection_proof-incongruence}
Equation (\ref{equation_congurence-S_l-1}) combined with \eqref{equation_cusppart} tell us that the $q$-series
\begin{equation}\label{equation_sptandp}
\eta(z)\cdot\sum_{n\geq 0}\left(\left(\frac{1-24n}{\ell}\right)-1\right)\left(12\spt(n)+(24n-1)p(n)\right)q^{n-\frac{1}{24}}
\end{equation}
is congruent modulo $\ell$ to a modular form of weight $\ell+1$. By multiplying \eqref{equation_sptandp} by $E_{\ell-1}^{\frac{\ell+1}{2}}\equiv 1\pmod{\ell}$, we can construct a modular form of weight $(\ell^2+1)/2$ that is congruent to \eqref{equation_sptandp} modulo $\ell$. Moreover, since $E_2\equiv E_{\ell+1}\pmod{\ell}$, we have
\begin{equation}\label{equation_partitionmol}
24\left(\frac{-24}{\ell}\right)\D^{\frac{\ell+1}{2}}\left(\frac{1}{\eta}\right)+\frac{E_{\ell+1}}{\eta}\cdot E_{\ell-1}^{\frac{\ell+1}{2}} \equiv\sum_{n\geq 0}\left(\left(\frac{1-24n}{\ell}\right)-1\right)(24n-1)p(n)q^{n-\frac{1}{24}}\pmod{\ell}.
\end{equation}
The differential operator $\D$ maps weight $k$ modular forms modulo $\ell$ to weight $k+\ell+1$ modular forms modulo $\ell$, it follows that the $q$-series \eqref{equation_partitionmol} multiplied by $\eta(z)$ is congruent to a modular form of weight $(\ell+1)^2/2$ modulo $\ell$. Thus
\[
\eta(z)\cdot\sum_{n\geq 0}\left(\left(\frac{1-24n}{\ell}\right)-1\right)\spt(n)q^{n-\frac{1}{24}}
\]
is congruent modulo $\ell$ to a modular form of weight $(\ell+1)^2/2$. Consequently, we may apply \cite[Theorem 1]{zbMATH07282550} (with $B=-1$) directly at this point to obtain the desired conclusion.

\subsection{Proof of Theorem \ref{theorem_inner-product}}\label{subsection_proof-inner-product}
For $\Re(s)>1$ and a normalized Hecke eigenform $f\in S_{2v+2}$, let
\[
I_s:=\langle[P_{-1,\frac{3}{2},s},\eta]_v,f\rangle.
\]
We shall demonstrate that $I_s$ admits an analytic continuation at $s=3/4$, and that its value at this point is precisely $24^{-v}\binom{2v}{v}\frac{\Gamma(1+2v)}{(\pi/3)^{1+2v}}\cdot D_f$. In conjunction with Proposition \ref{prop_poincare-analytic}, this establishes the proof of Theorem \ref{theorem_inner-product}. 

In the sequel, we employ the rising factorial which is defined by
\[
x^{\overline{n}}:=\begin{dcases}
x(x+1)\cdots (x+n-1)&\text{ if }n\geq 1,\\
1 &\text{ if }n=0,
\end{dcases}
\]
and falling factorial is defined as usual by
\[
x^{\underline{n}}:=\begin{dcases}
x(x-1)\cdots (x-n+1)&\text{ if }n\geq 1,\\
1 &\text{ if }n=0.
\end{dcases}
\]
We begin with a preparatory lemma.
\begin{lemma}\label{lemma_diff-of-phi}
We have
\begin{align*}
\D^j(\varphi_{-1,\frac{3}{2},s})=&(-24)^{-j}\sum_{0\leq i_1,i_2\leq j}\frac{j!}{i_1!i_2!(j-i_1-i_2)!}(-1)^{i_2}\frac{(s-\tfrac 34)^{\underline{i_1}}(s-\tfrac 34)^{\overline{i_2}}}{(2s)^{\overline{i_2}}}\\
&\times \left(\frac{\pi y}{6}\right)^{-i_1-\frac{i_2}{2}-\frac{3}{4}}M_{-\frac{3}{4}+\frac{i_2}{2},s-\frac{1}{2}+\frac{i_2}{2}}\left(\frac{\pi y}{6}\right)e\left(-\frac{x}{24}\right).
\end{align*}
\end{lemma}
\begin{proof}
We express $\varphi_{-1,\frac{3}{2},s}$ as a product of the following three functions and compute its derivative:
\begin{align*}
\D^j&\left(\left(\frac{\pi y}{6}\right)^{s-\frac{3}{4}}\cdot\left(\frac{\pi y}{6}\right)^{-s}M_{-\frac{3}{4},s-\frac{1}{2}}\left(\frac{\pi y}{6}\right)e^{-\frac{\pi y}{12}}\cdot q^{-\frac{1}{24}}\right)\\
&=\sum_{\substack{i_1,i_2,i_3\geq 0\\i_1+i_2+i_3=j}}\frac{j!}{i_1!i_2!i_3!}\D^{i_1}\left(\left(\frac{\pi y}{6}\right)^{s-\frac{3}{4}}\right)\D^{i_2}\left(\left(\frac{\pi y}{6}\right)^{-s}M_{-\frac{3}{4},s-\frac{1}{2}}\left(\frac{\pi y}{6}\right)e^{-\frac{\pi y}{12}}\right)\D^{i_3}\left(q^{-\frac{1}{24}}\right).
\end{align*}
With $\D f(y)=\frac{-1}{4\pi}f(y)$ for functions depending only on $y$, the first and third parts are computed as follows:
\begin{align*}
&\D^{i_1}\left(\left(\frac{\pi y}{6}\right)^{s-\frac{3}{4}}\right)=(-24)^{-i_1}\left(\frac{\pi}{6}\right)^{i_1}(s-\tfrac{3}{4})^{\underline{i_1}}\left(\frac{\pi y}{6}\right)^{s-\frac{3}{4}-i_1},\\
&\D^{i_3}q^{-\frac{1}{24}}=(-24)^{-i_3}q^{-\frac{1}{24}}.
\end{align*}
For the second part, we need the following property about the Whittaker functions (see\cite[13.15.19]{NIST:DLMF}):
\[
\frac{d^n}{dz^n}\left(e^{-z/2}z^{-\nu-1/2}M_{\mu,\nu}(z)\right)=(-1)^n\frac{(\mu+\nu+1/2)^{\overline{n}}}{(2\nu+1)^{\overline{n}}}e^{-z/2}z^{-\nu-n/2-1/2}M_{\mu+n/2,\nu+n/2},
\]
and we obtain
\begin{align*}
\D^{i_2}&\left(\left(\frac{\pi y}{6}\right)^{-s}M_{-\frac{3}{4},s-\frac{1}{2}}\left(\frac{\pi y}{6}\right)e^{-\frac{\pi y}{12}}\right)\\
&=(-24)^{-i_2}(-1)^{i_2}\frac{(s-\tfrac34)^{\overline{i_2}}}{(2s)^{\overline{i_2}}}\left(\frac{\pi y}{6}\right)^{-s-\frac{i_2}{2}}M_{-\frac{3}{4}+\frac{i_2}{2},s-\frac{1}{2}+\frac{i_2}{2}}\left(\frac{\pi y}{6}\right)e^{-\frac{\pi y}{12}}.
\end{align*}
Combining the above three parts, we obtain the proof of the lemma.
\end{proof}

From the definition of the Maass-Poincar\'e series $P_{-1,\frac{3}{2},s}(z)$ (\ref{equat_def-of-poincare}) and the Rankin-Cohen bracket (\ref{equation_def-RC-bracket}), the function $[P_{-1,\frac{3}{2},s},\eta]_v$ can be written as
\[
\sum_{\gamma\in\Gamma_{\infty}\setminus \SL_2(\mathbb{Z})}\sum_{0\leq j\leq v}(-1)^j\binom{v+\frac{1}{2}}{v-j}\binom{v-\frac{1}{2}}{j}\D^j(\varphi_{-1,\frac{3}{2},s})\D^{v-j}(\eta)\bigg|_{2+2v}\gamma.
\]
By Lemma \ref{lemma_RC-alpha-eta}, $[P_{-1,\frac{3}{2},s},\eta]_v$ equals 
\begin{align*}
(24)^{-v}&\sum_{0\leq j\leq v}\binom{v+\frac{1}{2}}{v-j}\binom{v-\frac{1}{2}}{j}
\sum_{0\leq i_1,i_2\leq j}\frac{j!}{i_1!i_2!(j-i_1-i_2)!}(-1)^{i_2}\frac{(s-\tfrac 34)^{\underline{i_1}}(s-\tfrac 34)^{\overline{i_2}}}{(2s)^{\overline{i_2}}}\\
&\times \sum_{n\geq 0}n^{2v-2j}\chi_{12}(n)\left(\left(\frac{\pi y}{6}\right)^{-i_1-\frac{i_2}{2}-\frac{3}{4}}M_{-\frac{3}{4}+\frac{i_2}{2},s-\frac{1}{2}+\frac{i_2}{2}}\left(\frac{\pi y}{6}\right)e\left(-\frac{x}{24}\right)q^{\frac{n^2}{24}}\right)\bigg|_{2+2v}\gamma.
\end{align*}
Using the unfolding method, and since the Fourier coefficients of eigenform are real, we obtain
\begin{align*}
I_s&=(24)^{-v}\sum_{0\leq j\leq v}\binom{v+\frac{1}{2}}{v-j}\binom{v-\frac{1}{2}}{j}
\sum_{0\leq i_1,i_2\leq j}\frac{j!}{i_1!i_2!(j-i_1-i_2)!}(-1)^{i_2}\frac{(s-\tfrac 34)^{\underline{i_1}}(s-\tfrac 34)^{\overline{i_2}}}{(2s)^{\overline{i_2}}}\\
&\times \sum_{n\geq 0}n^{2v-2j}\chi_{12}(n)\int_{0}^{\infty}\left(\frac{\pi y}{6}\right)^{-i_1-\frac{i_2}{2}-\frac{3}{4}}M_{-\frac{3}{4}+\frac{i_2}{2},s-\frac{1}{2}+\frac{i_2}{2}}\left(\frac{\pi y}{6}\right)e^{-2\pi(\tfrac{n^2}{24}+m)y}y^{2v}dy\\
&\times\int_{0}^{1}\sum_{m\geq 1}a_f(m)e\left(\frac{n^2-24m-1}{24}x\right)dx.
\end{align*}
The integral with respect to $x$ in the above expression vanishes unless $m =\frac{n^2-1}{24}$, yielding
\begin{align*}
I_s&=(24)^{-v}\sum_{0\leq j\leq v}\binom{v+\frac{1}{2}}{v-j}\binom{v-\frac{1}{2}}{j}
\sum_{0\leq i_1,i_2\leq j}\frac{j!}{i_1!i_2!(j-i_1-i_2)!}(-1)^{i_2}\frac{(s-\tfrac 34)^{\underline{i_1}}(s-\tfrac 34)^{\overline{i_2}}}{(2s)^{\overline{i_2}}}\\
&\times \sum_{n\geq 0}n^{2v-2j}\chi_{12}(n)\int_{0}^{\infty}\left(\frac{\pi y}{6}\right)^{-i_1-\frac{i_2}{2}-\frac{3}{4}}M_{-\frac{3}{4}+\frac{i_2}{2},s-\frac{1}{2}+\frac{i_2}{2}}\left(\frac{\pi y}{6}\right)e^{-\frac{\pi(2n^2-1)}{12}y}y^{2v}dy.
\end{align*}
We now employ the following formula (see \cite[13.23.1]{NIST:DLMF})
\begin{align*}
\int_0^{\infty}&e^{-zt}t^{\alpha-1}M_{\mu,\nu}(\beta t)dt\\
&=\frac{\beta^{\nu+1/2}\Gamma(\nu+\alpha+1/2)}{(z+\beta/2)^{\nu+\alpha+1/2}}{}_2F_1\left(\nu-\mu+1/2,\nu+\alpha+1/2;2\nu+1;\frac{2\beta}{\beta+2z}\right),
\end{align*}
where ${}_2F_1$ is the Gaussian hypergeometric function, thereby obtain
\begin{equation}\label{equation_expr-I_s}
\begin{aligned}
I_s&=(24)^{-v}\sum_{0\leq j\leq v}\binom{v+\frac{1}{2}}{v-j}\binom{v-\frac{1}{2}}{j}
\sum_{0\leq i_1,i_2\leq j}\frac{j!}{i_1!i_2!(j-i_1-i_2)!}(-1)^{i_2}\frac{(s-\tfrac 34)^{\underline{i_1}}(s-\tfrac 34)^{\overline{i_2}}}{(2s)^{\overline{i_2}}}\\
&\times \left(\frac{6}{\pi}\right)^{1+2v}\sum_{n\geq 1}n^{2v-2j}\chi_{12}(n) a_f\left(\frac{n^2-1}{24}\right)\frac{\Gamma(s-i_1+2v+\tfrac 14)}{n^{2(s-i_1+2v+\frac 14)}}\\
&\times{}_2F_1\left(s+\frac{3}{4},s-i_1+2v+\frac{1}{4};2s+i_2;\frac{1}{n^2}\right).
\end{aligned}
\end{equation}
We claim that the series in $n$ in (\ref{equation_expr-I_s}) is holomorphic at 3/4.
Consequently, $I_s$ is defined at $s=3/4$, and by analytic continuation equals $\langle [F,\eta]_v,f\rangle$. To establish our claim, it suffices to show that the series
\[
\sum_{n\geq 1}\frac{\chi_{12}(n) a_f(\frac{n^2-1}{24})}{n^{2(s+v+\tfrac{1}{4})}}
\]
converges at $s=3/4$. Choose $1\leq M_n\leq 6$ with $M_n\equiv\frac{n+1}{2}$ if $n\not\equiv 7\pmod {12}$ and $M_n=2$ if $n\equiv 7\pmod {12}$. We factor $\frac{n^2-1}{24}$ as a product of two coprime positive integers:
\[
\frac{n^2-1}{24}=\frac{n+1}{2M_n}\cdot\frac{(n-1)M_n}{12}.
\]
Since $f$ is a Hecke eigenform, we have
\[
a_f\left(\frac{n^2-1}{24}\right)=a_f\left(\frac{n+1}{2M_n}\right)a_f\left(\frac{(n-1)M_n}{12}\right).
\]
Now we use a result of Blomer \cite[Theorem 1.3]{zbMATH02140876} and choose  a suitable test function as in the proof of \cite[Corollary 1.4]{zbMATH02140876}. We then have the estimate
\[
\sum_{1\leq n\leq x}\chi_{12}(n)a_f\left(\frac{n+1}{2M_n}\right)a_f\left(\frac{(n-1)M_n}{12}\right)\ll x^{2v+2-\delta}
\]
with some $\delta>0$, which via partial summation, establishes our claim. 

As $s \to 3/4$, the presence of the factor $s-3/4$ in (\ref{equation_expr-I_s}) ensures that all sums vanish except in the case $i_1=i_2=0$. Noting that
\[
{}_2F_1\left(\frac{3}{2},1+2v;\frac{3}{2};\frac{1}{n^2}\right)=\left(\frac{n^2}{n^2-1}\right)^{1+2v}.
\]
We have
\begin{equation}\label{equation_inner-D-e-f}
\begin{aligned}
\langle [F,\eta]_v,f\rangle&=24^{-v}\left(\frac{6}{\pi}\right)^{1+2v}\Gamma(1+2v)\sum_{0\leq j\leq v}\binom{v+\frac{1}{2}}{v-j}\binom{v-\frac{1}{2}}{j}\\
&\times \sum_{n\geq 1}\frac{a_f(\frac{n^2-1}{24})\chi_{12}(n)}{n^{2+2v+2j}}\left(\frac{n^2}{n^2-1}\right)^{1+2v}.
\end{aligned}
\end{equation}
Apply (\ref{equation_com-ident}) together with 
\[
\sum_{0\leq j\leq v}\binom{2v+1}{2j+1}\frac{2}{n^{1+2j}}=\left(1+\frac{1}{n}\right)^{1+2v}-\left(1-\frac{1}{n}\right)^{1+2v},
\]
and obtain
\begin{equation}\label{equation_last-ident}
\sum_{0\leq j\leq v}\binom{v+\frac{1}{2}}{v-j}\binom{v-\frac{1}{2}}{j}\frac{1}{n^{2j+1}}=2^{-2v-1}\binom{2v}{v}\left(\left(1+\frac{1}{n}\right)^{1+2v}-\left(1-\frac{1}{n}\right)^{1+2v}\right).
\end{equation}
Insertion of (\ref{equation_last-ident}) into (\ref{equation_inner-D-e-f}), followed by straightforward computations, completes the proof of the theorem.

\providecommand{\bysame}{\leavevmode\hbox to3em{\hrulefill}\thinspace}
\providecommand{\MR}{\relax\ifhmode\unskip\space\fi MR }
\providecommand{\MRhref}[2]{%
  \href{http://www.ams.org/mathscinet-getitem?mr=#1}{#2}
}
\providecommand{\href}[2]{#2}

\end{document}